\documentclass[10pt]{article}

\usepackage{graphics}
\usepackage{color}
\usepackage{url}
\usepackage{subfigure}
\usepackage{latexsym}
\usepackage{amsfonts,amssymb,amsmath}
\usepackage[dvips]{graphicx}
\usepackage{epsfig}
\usepackage{verbatim}   % provides comment and verbatim commands that work!
\usepackage{amsmath}
\usepackage{amssymb}
\usepackage{fancyhdr}
\usepackage[driver=pdftex]{geometry}
\usepackage{algorithm}
\usepackage{algpseudocode}
\usepackage[utf8]{inputenc}
\usepackage{amsthm}

\usepackage{mathrsfs}
\usepackage[normalem]{ulem}
\allowdisplaybreaks

\newcommand{\R}{\mathbb{R}}
\newcommand{\C}{\mathbb{C}}
\newcommand{\N}{\mathbb{N}}
\newcommand{\xv}{\mathbf{x}}
\newcommand{\bv}{\mathbf{b}}
\newcommand{\tsteps}{m}
\newcommand{\txv}{\tilde{\mathbf{x}}}
\newcommand{\uv}{\mathbf{u}}
\newcommand{\yv}{\mathbf{y}}
\newcommand{\wv}{\mathbf{w}}
\newcommand{\wvh}{\hat{\mathbf{w}}}
\newcommand{\vb}{\mathbf{v}}
\newcommand{\zv}{\mathbf{z}}
\newcommand{\tzv}{\tilde{\mathbf{z}}}
\newcommand{\zvh}{\hat{\mathbf{z}}}
\newcommand{\F}{\mathbf{F}}
\newcommand{\tF}{\tilde{\mathbf{F}}}
\newcommand{\nx}{n}
\newcommand{\nt}{N}
\newcommand{\tra}{\prime}
\newcommand{\A}{\mathbf{A}}
\newcommand{\B}{\mathbf{B}}
\newcommand{\Id}{\mathbf{I}}
\newcommand{\lam}{\lambda}
\newcommand{\Or}{\mathbb{O}}
\newcommand{\Oz}{\mathbf{0}}

\newtheorem{theorem}{Theorem}
\newtheorem{lemma}{Lemma}

\newtheorem{problem}{Problem}
\newtheorem{remark}{Remark}
\newtheorem{corollary}{Corollary}

\usepackage[hidelinks]{hyperref} % so that clicking on citation number takes you to the reference page

\begin{document}
\title{An efficient quantum algorithm for simulating polynomial differential equations}
\author{Amit Surana, Abeynaya Gnanasekaran and Tuhin Sahai}
\maketitle

\begin{abstract}

In this paper, we present an efficient quantum algorithm to simulate nonlinear differential equations with polynomial vector fields of arbitrary degree on quantum platforms.  Models of physical systems that are governed by ordinary differential equations (ODEs) or partial differential equation (PDEs) arise extensively in science and engineering applications and can be challenging to solve on classical computers due to high dimensionality, stiffness, nonlinearities, and sensitive dependence to initial conditions. For sparse $n$-dimensional linear ODEs, quantum algorithms have been developed which can produce a quantum state proportional to the solution in $\textsf{poly}(\log(\nx))$) time using the quantum linear systems algorithm (QLSA). Recently, this framework was extended to systems of nonlinear ODEs with quadratic polynomial vector fields by applying Carleman linearization that enables the embedding of the quadratic system into an approximate linear form. A detailed complexity analysis was conducted which showed significant computational advantage under certain conditions. We present an extension of this algorithm to deal with systems of nonlinear ODEs with $k$-th degree polynomial vector fields for arbitrary (finite) values of $k$. The steps involve: 1) mapping the $k$-th degree polynomial ODE to a higher dimensional quadratic polynomial ODE; 2) applying Carleman linearization to transform the quadratic ODE to an infinite-dimensional system of linear ODEs; 3) truncating and discretizing the linear ODE and solving using the forward Euler method and QLSA. Alternatively, one could apply Carleman linearization directly to the $k$-th degree polynomial ODE, resulting in a system of infinite-dimensional linear ODEs, and then apply step 3. This solution route can be computationally more efficient. We present detailed complexity analysis of the proposed algorithms and prove polynomial scaling of runtime on $k$. We demonstrate the computational framework on a numerical example.
\end{abstract}

\section{Introduction}
Models of physical systems that are governed by ordinary differential equations (ODEs) or partial differential equation (PDEs) arise extensively in science and engineering applications. Examples include nonlinear ODEs in molecular dynamics and epidemiology; Euler and Navier-Stokes PDEs in fluid dynamics, weather forecasting, plasma magnetohydrodynamics, and astrophysics; reaction-diffusion PDEs in chemistry, biology and ecology to name a few. Despite tremendous progress in the state-of-the-art algorithms for solving these equations on classical platforms, these methods are unable to address challenges related to curse-of-dimensionality, highly nonlinear dynamics, and strongly coupled degrees of freedom over multiple length- and time- scales. Consequently, quantum algorithms, due to inherent potential advantages that typically result in polynomial- to exponential- acceleration over their classical counterparts, could be key enablers to deal with some of these challenges. However, quantum algorithms for simulating dynamical systems governed by ODEs/PDEs are still in a nascent stage of development and will require substantial progress before they can be applicable to a wide range of systems. In this work, we significantly expand the domain of applicability of these methods by pushing the state-of-the-art quantum algorithms to arbitrary (finite) degree polynomial ODEs. Our algorithms maintain their exponential advantage over classical methods under prescribed conditions. 

For linear differential equations, it has been demonstrated that quantum algorithms offer the prospect of rapidly characterizing the solutions of high-dimensional systems of linear ODEs \cite{berry2014high,berry2017quantum,childs2020quantum,krovi2022improved}, and PDEs \cite{childs2021high,costa2019quantum,linden2020quantum,montanaro2016quantum}. These algorithms produce a quantum state proportional to the solution of a sparse (or block-encoded) $\nx$-dimensional system of linear differential equations in time $\textsf{poly}(\log(\nx))$) using the quantum linear system algorithm (QLSA). For nonlinear differential equations a variety of different frameworks have been explored. For ODEs with polynomial nonlinearities, a quantum algorithm was proposed which simulates the system by storing multiple copies of the solution~\cite{leyton2008quantum}. The complexity of this approach is $\textsf{poly}(\log(\nx))$) in dimension but exponential in the evolution time $T$, scaling as $\mathcal{O}(\epsilon^{-T})$ (where, $\epsilon$  is allowed error in the solution). This is consequence of the requirement that one has to use an exponentially increasing number of copies to accurately capture the nonlinearity. Recently, by applying Carleman linearization it was shown that the exponential time dependence can be reduced to polynomial scaling in $T$ in certain settings \cite{liu2021efficient}, thereby providing an exponential improvement for those use cases. Here the quantum algorithm is based on applying truncated Carleman linearization to quadratic ODEs to transform them into a finite set of linear ODEs, and solving them using a combination of forward Euler numerical method and QLSA. The authors performed a detailed complexity analysis based on an improved convergence theorem for Carleman linearization, a bound for the global error of the Euler method, an upper bound on the condition number of the linear algebraic equations, and a lower bound on the success probability of the final measurement. Specifically, assuming $R_2<1$, where $R_2$ is a parameter characterizing the ratio of the nonlinearity and forcing to the linear dissipation, the query/gate complexity of proposed algorithm takes the form $\frac{sT^2 q}{\epsilon}\textsf{poly}(\log T, \log \nx, \log 1/\epsilon)$, where, $q=\|\xv(0)\|/\|\xv(T)\|$ measures decay of the solution $\xv(t)$, and $s$ is sparseness of system matrices describing the ODE. Alternate techniques based on linear representation of dynamical systems such as Koopman-von Neumann mechanics and Lioville approaches \cite{jin2022time,joseph2020koopman,lin2022koopman}, and Koopman framework for ergodic dynamical systems \cite{giannakis2022embedding} have been proposed for simulating nonlinear ODEs on quantum platform. For a comparison of advantages and disadvantages of these approaches, we refer the reader to \cite{jin2022time,lin2022koopman}.

In this paper, we generalize the Carleman linearization based framework, developed in \cite{liu2021efficient}, for simulating quadratic polynomial ODEs on quantum computers  to nonlinear ODEs with $k$-th degree polynomial vector fields for arbitrary (finite) values of $k$. The steps involve: 1) mapping the $k$-th degree polynomial ODE to a higher-dimensional quadratic polynomial ODE; 2) applying Carleman linearization to transform the quadratic ODE to an infinite-dimensional system of linear ODEs; 3) truncating and discretizing the linear ODE and then finally solving it using the forward Euler method and QLSA. Leveraging the analysis presented in \cite{liu2021efficient}, we  also present a complexity analysis of the proposed algorithm, proving that under the assumption that $R_k<1$ (which is a higher order generalization of $R_2$), the query/gate complexity of the proposed algorithm takes the form $\frac{sk(k-1)T^2 q_k}{\epsilon}\textsf{poly}(\log T, (k-1)\log \nx, \log 1/\epsilon)$, where $s,\epsilon,T$ are as stated above, and $q_k=\frac{\sqrt{\sum_{i=1}^{k-1}\|\xv(0)\|^{2i}}}{\sqrt{\sum_{i=1}^{k-1}\|\xv(T)\|^{2i}}}$  measures the decay of the solution. Thus, our analysis generalizes the complexity analysis result from \cite{liu2021efficient}, showing an explicit polynomial dependence on the degree $k$ of the polynomial vector field.

The outline of paper is as follows. In Sections~\ref{mathprelim} and~\ref{sec:carl}, we cover mathematical preliminaries and review Carleman linearization framework that includes a general approach to transforming arbitrary $k$-th degree polynomial ODE systems to higher dimensional quadratic polynomial ODE systems, respectively. We review the complexity analysis results for quadratic ODE systems from \cite{liu2021efficient} in Section \ref{sec:QCODE}, and present an analogous analysis for $k$-th degree polynomial ODE systems. Numerical results are presented in Section \ref{sec:num}, followed by conclusions and proposed future work in Section \ref{sec:conc}.

\section{Mathematical Preliminaries}\label{mathprelim}
Let $\N = \{1, 2,\dots\}$, $\R$, and $\C$ be the sets of positive integers,  real numbers, and  complex numbers respectively. Let $\Id_{n\times n}$
denote the identity matrix of order $n$. Hereafter, $\|\cdot\|_p,p=1,2,\dots,\infty$ denotes the $l_p$-norm in the Euclidean space $\R^n$,
\begin{equation}\label{eq:normx}
\|\xv\|_p=\left(\sum_{j=1}^n|x_j|^p\right)^{1/p},\quad \xv\in \R^n.
\end{equation}
For the norm of a matrix $\A =(a_{ij})\in \R^{n\times m}$, we use the induced norm, namely
\begin{equation}\label{eq:normA}
\|\A\|_p=\max_{\xv \in \R^m} \frac{\|\A \xv\|_p}{\|\xv\|_p}.
\end{equation}
We will use $p=2$, i.e., $l_2$ norm for vectors and spectral norm for matrices if not stated otherwise. In the following, $\A^\tra$ will denote the transpose of $\A$.
%Note that, for $p=\infty$
%\begin{equation}\label{eq:normx}
%\|\xv\|_\infty=\max_{i}|x_i|,\qquad \|\A\|_\infty=\max_{i}\sum_{j=1}^n|a_{ij}|.
%\end{equation}
%which is the maximum absolute row sum of the matrix.

For any pair of vectors, $\xv\in \R^n$ and $\yv \in \R^m$, their Kronecker product $\wv\in\R^{nm}$  is,
\begin{equation}\label{eq:kron}
\wv=\xv\otimes\yv= (x_1y_1, x_1y_2, . . . , x_1y_m, x_2y_1, . . . , x_2y_m, . . . , x_ny_1, . . . , x_ny_m)^\tra.
\end{equation}
We use an analogous definition for matrices. If  $\A\in\R^{m\times n}$ and $\mathbf{B}\in\R^{p\times q}$, then $\mathbf{C}\in \R^{mp\times nq}$ is
\begin{equation}\label{eq:kronM}
\mathbf{C}=\A\otimes\mathbf{B}=\left(
                                 \begin{array}{ccc}
                                   a_{11}\mathbf{B} & \cdots & a_{1n}\mathbf{B} \\
                                   \vdots & \vdots & \vdots \\
                                    a_{m1}\mathbf{B} & \cdots & a_{mn}\mathbf{B} \\
                                 \end{array}
                               \right).
\end{equation}
%and
%\begin{equation}\label{eq:normKron}
%\|\A\otimes\mathbf{B}\|_2=\|\A\|_2 \|\mathbf{B}\|_2.
%\end{equation}
\begin{lemma}
Let $\A\in\R^{m\times n}$ and $\mathbf{B}\in\R^{p\times q}$, then the spectral norm of their Kronecker product satisfies,
\begin{equation}\label{eq:normKron}
\|\A\otimes\mathbf{B}\|_2=\|\A\|_2 \|\mathbf{B}\|_2.
\end{equation}
\end{lemma}
\begin{proof}
Since, the singular values of the Kronecker product of two matrices is the product of their singular values (Theorem 4.2.15 in~\cite{horn_johnson_1991}) and $\|A\|_2 = \sigma_{\max}(A)$, the result follows immediately.
%\ag{Maybe saying this is more straightforward -- The singular values of the Kronecker product of two matrices is the product of their singular values (Theorem 4.2.15 in~\cite{horn_johnson_1991}) and $\|A\|_2 = \sigma_{\max}(A)$. }\textcolor{red}{Yes, this is cleaner I cannot find the book. I found one version but that did not have a theorem 4.2.15.}
%The eigenvalues of the Kronecker product of two matrices is the product of their of eigenvalues~\cite{broxson2006kronecker}. One can then use the following identities, $(\mathbf{A}\otimes\mathbf{B})^\tra=\mathbf{A}^\tra\otimes \mathbf{B}^\tra$ and $(\mathbf{A}\otimes \mathbf{B})(\mathbf{C}\otimes \mathbf{D})=(\mathbf{A}\mathbf{C})\otimes(\mathbf{B}\mathbf{D})$, to prove the equality in (\ref{eq:normKron}).
\end{proof}

The Kronecker power is a convenient notation to express all possible products of elements
of a vector up to a given order, and is denoted by,
\begin{equation}\label{eq:powx}
\xv^{[i]}=\underbrace{\xv\otimes\xv\cdots\otimes\xv}_{i-\textsf{times}},
\end{equation}
for any $i\in\N$ with the convention $\xv^{[0]}=1$. Moreover, $\mbox{dim} (\xv^{[i]})= n^i$,  and each component of $\xv^{[i]}$ is of the form $x_1^{\omega_1}x_2^{\omega_2}\cdots x_n^{\omega_n} $ for some multi-index $\mathbf{\omega}\in \N^n$  of weight $\sum_{j=1}^n\omega_j=i$.
Similarly, we denote the matrix Kronecker power by,
\begin{equation}\label{eq:powA}
\A^{[i]}=\underbrace{\A\otimes\A\dots\otimes\A}_{i-\textsf{times}}.
\end{equation}

\section{Carleman Linearization}\label{sec:carl}
In this section, we summarize the Carleman framework for linearization of nonlinear systems. Our summary is adapted from the following references \cite{amini2022carleman,forets2017explicit,kowalski1991nonlinear}.

\subsection{Higher Order Polynomials}\label{sec:carlHO}
Consider an initial value problem (IVP) for a $k-$th order inhomogeneous polynomial ODE,
\begin{eqnarray}
\dot{\xv}(t)&=&\F_0+\F_1\xv+\F_2\xv^{[2]}+\cdots+\F_k\xv^{[k]},\label{eq:gen}\\
\xv(0)&=&\xv_0\in\R^\nx,\notag
\end{eqnarray}
where, $\xv\in \R^\nx$ is an $n$-dimensional vector and $\F_i\in\R^{\nx\times \nx^i}$ are time independent matrices. Note that the index $i$ runs from $0$ to the highest polynomial degree i.e. $i=0,\hdots,k$.  Then, $\zv_i=\xv^{[i]},\,\, \text{a column vector}, \forall i \neq 0$ satisfies,
\begin{equation}\label{eq:genz}
\dot{\zv}_i=\sum_{j=0}^{k}\A^{i}_{i+j-1}\zv_{i+j-1},
\end{equation}
where, $\A^{i}_{i+j-1}\in\R^{\nx^i\times \nx^{i+j-1}}$ is given by,
\begin{equation}\label{eq:Adef}
\A^{i}_{i+j-1}=\sum_{p=1}^{i}\overbrace{\Id_{\nx\times \nx}\otimes\cdots\otimes\underbrace{\F_j}_{p-\textsf{th position}}\otimes\cdots\otimes \Id_{\nx\times \nx}}^{i \textsf{ factors}}.
\end{equation}
The matrices $\A^{i}_{i+j-1}$ satisfy the following recurrence relation (see Proposition 3.3 in \cite{forets2017explicit}),
\begin{equation}\label{eq:req}
\A^{i}_{i+j-1}=\A^{i-1}_{i+j-2}\otimes \Id_{\nx\times\nx}+\Id_{\nx\times\nx}^{[i-1]}\otimes \A^1_j,
\end{equation}
for $i\geq 2, 0 \leq j\leq k$. Note that for $i=1$, $\A^1_j=\F_j,\, \forall j$. 
\begin{lemma}\label{lemmanrm}
The spectral norm of matrices $\A^{i}_{i+j-1}$ satisfy
\begin{equation}\label{eq:Anorm}
\|\A^{i}_{i+j-1}\|_2\leq i\|\F_{j}\|_2.
\end{equation}
\end{lemma}
\begin{proof}
By definition (\ref{eq:Adef}),
\begin{eqnarray}\label{eq:nr}
\|\A^{i}_{i+j-1}\|_2 &\leq& \sum_{p=1}^{i}\|\overbrace{\Id_{\nx\times \nx}\otimes\cdots\otimes\underbrace{\F_j}_{p-\textsf{th position}}\otimes\cdots\otimes \Id_{\nx\times \nx}}^{i \textsf{ factors}}\|_2,\notag\\
&\leq &\sum_{p=1}^{i}\|\F_{j}\|_2= i\|\F_{j}\|_2,
\end{eqnarray}
where we have used the triangle inequality,  equality (\ref{eq:normKron}), and the fact that $\|\Id_{\nx\times\nx}\|_2=1$. 
\end{proof}

%Also, their norms satisfy the following relation,
%\begin{eqnarray}\label{eq:nr}
%\|\A^{i}_{i+j-1}\|_2 &\leq& \sum_{p=1}^{i}\|\overbrace{\Id_{\nx\times \nx}\otimes\cdots\otimes\underbrace{\F_j}_{p-\textsf{th position}}\otimes\cdots\otimes \Id_{\nx\times \nx}}^{i \textsf{ factors}}\|_2,\notag\\
%&\leq &\sum_{p=1}^{i}\|\F_{j}\|_2= i\|\F_{j}\|_2,
%\end{eqnarray}
%where we have used the triangle inequality, definition (\ref{eq:Adef}), equality (\ref{eq:normKron}), and the fact that $\|\Id_{\nx\times\nx}\|_2=1$. 
%\textcolor{red}{This is derived in reference 10, should we present this in detail here?}
Let $\zv=(\zv_1^\tra,\zv_2^\tra,\hdots)^\tra$ be an infinite dimensional vector, 
%resulting from the Carleman linearization. Eqn.~(\ref{eq:gen}) implies that 
then $\zv$ satisfies an infinite dimensional set of linear ODEs of the form,
\begin{equation}\label{eq:carl}
\dot{\zv}=\mathcal{A}\zv+\mathcal{F}_0,
\end{equation}
with an initial condition $\zv(0)=(\xv(0)^{[1]},\xv(0)^{[2]},\hdots)^\tra$, where,
\begin{equation*}
\mathcal{A}=\left(
              \begin{array}{cccccccc}
                \A_1^1 & \A^1_2 & \A^1_3 & \cdots & \A_k^1 & 0 & 0 & \cdots \\
                \A_1^2 & \A_2^2 & \A_3^2 & \cdots  & \A^2_k  & \A^2_{k+1} & 0 & \cdots \\
                0 & \A^3_2 & \A_3^3 & \cdots & \A^3_k & \A^3_{k+1} & \A_{k+2}^3 & \cdots \\
                \vdots & \vdots & \vdots & \vdots & \vdots & \vdots & \vdots & \vdots \\
              \end{array}
            \right),\quad \mathcal{F}_0=\left(\begin{array}{c}
                                          \F_0 \\
                                          0 \\
                                          0 \\
                                          \vdots
                                        \end{array}\right).
\end{equation*}
The infinite dimensional system (\ref{eq:carl}) is known as the Carleman linearization (CL) of (\ref{eq:gen}). In practice, the above system needs to be truncated to a finite order $\nt$ that leads to a closure problem. Here we follow the standard null closure scheme which consists of eliminating the dependence on variables of order exceeding $\nt$. The $\zv_i$ variables removed in this truncation procedure start at $N-k+2$ resulting in,
\begin{eqnarray}\label{eq:carlf1}
\dot{\zvh}_i&=&\sum_{j=0}^{k}\A^{i}_{i+j-1}\zvh_{i+j-1},\quad i=1,\hdots,N-k+1,\\
\dot{\zvh}_i&=&\sum_{j=0}^{l}\A^{i}_{i+j-1}\zvh_{i+j-1},\quad i=N-k+2,\hdots,N,
\end{eqnarray}
where, $l=\min\{k,N+1-i\}$. If we consider the finite-dimensional vector $\zvh=(\zvh_1^\tra,\zvh_2^\tra,\hdots,\zvh_\nt^\tra)^\tra$, we can write the above system in matrix form leading to an IVP:
\begin{eqnarray}
\dot{\zvh}&=&\A_\nt\zvh+\F_\nt, \label{eq:carlf}\\
\zvh_i(0)&=&\xv(0)^{[i]},\quad \forall i=1,\hdots,\nt,\notag
\end{eqnarray}
where, the initial condition is compatible with that for system (\ref{eq:carl}), $\F_N=(\F_0^\tra,0,\dots,0)^\tra\in \R^{\nx_k}$ with $\nx_k=(\nx^{\nt+1}-\nx)/(\nx-1)$, and  $\A_\nt$ is a finite-dimensional $\nx_k \times \nx_k$ block upper-Hessenberg matrix.
%Since $\textsf{dim }\A^i_{i+j-1}=\nx^i\times \nx^{i+j-1}$, the size of $\mathcal{A}_\nt$ is $n_x\time \n_x$.

\subsection{Conversion to Quadratic ODE}\label{sec:conv}
Introducing $\tilde{\xv}=(\xv^\tra,\dots,(\xv^{[k-1]})^\tra)^\tra$ and following the steps in Proposition 3.4 in \cite{forets2017explicit} (see also the Lemma \ref{appexLem} in the Appendix), the $k-$th order polynomial system (\ref{eq:gen}) can be transformed into the following quadratic form,
%(see Proposition 3.4 in \cite{forets2017explicit}),
\begin{equation}\label{eq:qd}
\dot{\tilde{\xv}}=\tF_0+\tF_1\tilde{\xv}+\tF_2\tilde{\xv}^{[2]},
\end{equation}
where,
\begin{equation}\label{eq:tF1}
\tF_0=\left(\begin{array}{c}
                                          \F_0 \\
                                          0 \\
                                          0 \\
                                          \vdots\\
                                          0
                                        \end{array}\right), \quad \tF_1=\left(
        \begin{array}{ccccc}
          \A^{1}_{1} & \A^{1}_{2}& \A^{1}_{3} & \cdots & \A^{1}_{k-1} \\
          \A^{2}_1 & \A^{2}_{2} & \A^{2}_{3} & \cdots & \A^{2}_{k-1} \\
          0 & 0 & \cdots & \cdots & \cdot  \\
          \vdots & \vdots  & \vdots  & \vdots  & \vdots  \\
          0 & 0 & \cdots & \A^{k-1}_{k-2} & \A^{k-1}_{k-1} \\
        \end{array}
      \right),
\end{equation}
and,
\begin{equation}\label{eq:tF2}
\tF_2=\left(
        \begin{array}{ccccccccccccc}
          0 & \cdots & 0 & \A^1_k & 0 & \cdots& 0 & 0 & 0 & \cdots & \cdots & 0 & 0 \\
          0 & \cdots & 0 & \A^2_k & 0 & \cdots & 0 & \A^2_{k+1} & 0 & \cdots & \cdots & 0 & 0\\
          \vdots &  \vdots  &  \vdots & \vdots & \vdots & \vdots & \vdots & \vdots & \vdots & \vdots & \vdots & \vdots & \vdots \\
          0 & \cdots & 0 & \A^{k-1}_k & 0 & \cdots & 0 & \A^{k-1}_{k+1} & 0 & \cdots & \cdots & 0 & \A^{k-1}_{2(k-1)}\\
        \end{array}
      \right).
\end{equation}

\begin{lemma}\label{lem0}
The norm of $\txv_0 = (\xv_0^\tra, (\xv_0^{[2]})^\tra, \dots, (\xv_0^{[k-1]})^\tra)^\tra$ can be expressed as,
\begin{equation}\label{eq:tx0}
\|\txv_0\|_2 =\sqrt{\sum_{i=1}^{k-1} \|\xv_0\|_2^{2i}},
\end{equation}
%\label{eq:tx0} \] , %< \frac{\|x_0\|_2}{\sqrt{1-\|x_0\|_2^2}} \text{ if } \| x_0\|_2 < 1\]
and, the norm of $\tF_2$ as defined in (\ref{eq:tF2}) satisfies,
\begin{equation}
\| \tF_2 \|_2 \leq (k-1)\sum_{i=2}^{k} \|\F_{i} \|_2. \label{eq:nrF2}
\end{equation}
\end{lemma}

\begin{proof}
To show (\ref{eq:tx0}), we use the equality (\ref{eq:normKron}) leading to
%We have, $\|\txv_0\|_2 = (\xv_0, \xv_0^{[2]}, \dots, \xv_0^{[k-1]})'$
\begin{align*}
\|\txv_0\|_2 &= \sqrt{\sum_{i=1}^{k-1} \|\xv_0^{[i]}\|_2^2}=\sqrt{\sum_{i=1}^{k-1} \|\xv_0\|_2^{2i}}. %\\
%&= \sqrt{\sum_{i=1}^{k-1} \|\xv_0\|_2^{2i}}.
%& \leq \sqrt{\frac{\|\xv_0\|_2^2}{1-\|x_0\|_2^2}}
\end{align*}

In order to prove (\ref{eq:nrF2}), permute the columns of $\tF_2$ such that,
\[
\tF_2 \mathbf{P} = \begin{bmatrix}
\A_k^1 & 0 & \dots & 0 & 0 & \dots & 0 \\
\A_k^2 & \A_{k+1}^2 & \dots & 0 & 0 & \dots & 0 \\
\vdots & \vdots & \vdots & \vdots & 0 & \dots & 0 \\
\A_k^{k-1} & \A_{k+1}^{k-1} & \dots & \A_{2(k-1)}^{k-1} & 0 & \dots & 0
\end{bmatrix},
\]
where, $\mathbf{P}$ is the corresponding permutation matrix. It then follows, that
\begin{align*}
\|\tF_2\|_2 = \|\tF_2\mathbf{P}\|_2 &= \Bigg \| \begin{bmatrix}
\A_k^1 & 0 & \dots & 0  \\
\A_k^2 & \A_{k+1}^2 & \dots & 0  \\
\vdots & \vdots & \vdots & \vdots  \\
\A_k^{k-1} & \A_{k+1}^{k-1} & \dots & \A_{2(k-1)}^{k-1}
\end{bmatrix} \Bigg \|_2, \notag \\
&\leq  \Bigg \| \begin{bmatrix}
\A_k^1 & 0 & \dots & 0  \\
0 & \A_{k+1}^2 & \dots & 0  \\
\vdots & \vdots & \vdots & \vdots  \\
0 & 0 & \dots & \A_{2(k-1)}^{k-1}
\end{bmatrix} \Bigg \|_2 + \Bigg \| \begin{bmatrix}
\A_k^2 & 0 & \dots & 0  \\
\A_k^3 & \A_{k+1}^3 & \dots & 0 \\
\vdots & \vdots & \vdots & \vdots  \\
\A_k^{k-1} & \A_{k+1}^{k-1} & \dots & \A_{2(k-1)-1}^{k-1}
\end{bmatrix} \Bigg \|_2, \notag\\
&\leq \sum_{i=1}^{k-1} \Bigg \| \begin{bmatrix}
\A_k^i & 0 & \dots & 0  \\
0 & \A_{k+1}^{i+1} & \dots & 0  \\
\vdots & \vdots & \vdots & \vdots  \\
0 & 0 & \dots & \A_{2(k-1)-(i-1)}^{k-1}
\end{bmatrix} \Bigg \|_2, \\
&= \sum_{i=1}^{k-1} \max \{ \|\A^i_k\|_2, \|\A^{i+1}_{k+1}\|_2, \dots, \|\A_{2(k-1)-(i-1)}^{k-1}  \|_2 \}.
\end{align*}
Using the Lemma~\ref{lemmanrm} %bound~(\ref{eq:nr}) 
in the previous expression leads to the desired result,
\begin{align*}
    \|\tF_2\|_2 &\leq \sum_{i=1}^{k-1} \max \{i \|\F_{k-(i-1)}\|_2, (i+1) \|\F_{k-(i-1)}\|_2, \dots, (k-1) \|\F_{k-(i-1)}\|_2  \},\notag \\
    &= (k-1)\sum_{i=1}^{k-1} \|\F_{k-(i-1)}\|_2=(k-1)\sum_{i=2}^{k} \|\F_{i} \|_2.\notag
\end{align*}
A similar result was proven in \cite{forets2017explicit} for the infinity norm.

\end{proof}

%Moreover, the supremum norm of the linear and quadratic parts satisfy, respectively
%\begin{equation}\label{eq:nrF1}
%\|\tF_1\|_\infty\leq \max_{1\leq i\leq k-1}\sum_{j=1}^i\|\F_j\|_\infty,
%\end{equation}
%and,
%\begin{equation}\label{eq:nrF2}
%\|\tF_2\|_\infty\leq (k-1)\sum_{j=2}^k\|\F_j\|_\infty.
%\end{equation}

\section{Carleman Linearization based Quantum Algorithms}\label{sec:QCODE}
In this section, we review CL based framework for simulating quadratic polynomial ODEs on quantum computers developed in \cite{liu2021efficient}, and generalize the approach to higher order polynomial ODEs.

\subsection{Quadratic ODE}\label{sec:QCquadODE}
Consider a system of inhomogeneous quadratic polynomial ODEs, i.e. system (\ref{eq:gen}) with $k=2$,
\begin{eqnarray}
\dot{\xv}&=&\F_0+\F_1\xv+\F_2\xv^{[2]}, \label{eq:qdc}\\
\xv(0)&=&\xv_0\in\R^\nx.\notag
\end{eqnarray}
For the above system, CL when truncated to order $\nt$, results in a finite system of linear ODEs (see Eq.~(\ref{eq:carlf})),
\begin{equation}\label{eq:carlfq}
\dot{\zvh}=\A_\nt\zvh+\bv,
\end{equation}
where, $\zvh=(\zvh_1^\tra,\zvh_2^\tra,\cdots,\zvh_\nt^\tra)^\tra$, and
\begin{equation}\label{eq:AN2}
\A_\nt=\left(\begin{array}{cccccc}
                \A_1^1 & \A^1_2 & 0 & \cdots & 0 & \cdots \\
                \A^2_1 & \A_2^2 & \A^2_3 & \cdots & 0 & \cdots \\
                0 & \A^3_2 & \A_3^3 & \A^3_4 & 0 & \cdots \\
                \vdots & \vdots & \vdots & \vdots & \vdots &\vdots  \\
                0 & 0 & 0 & 0 & \A^{N}_{N-1} & \A^{N}_N\\
              \end{array}
            \right),\quad \bv=\left(\begin{array}{c}
                                          \F_0 \\
                                          0 \\
                                          0\\
                                          \vdots \\
                                          0
                                        \end{array}\right).
\end{equation}

%$\F_0(t)\in \R^\nx$ is a $C^1$ continuous function of $t$, and
%\begin{equation}\label{eq:carlh}
%\dot{\zv}=\mathcal{A}^h\zv+\mathcal{F}_0(t),
%\end{equation}
%where,
%\begin{equation}\label{eq:carlh}
%\mathcal{A}^h=\left(
%              \begin{array}{cccccc}
%                \A_1^1 & \A^1_2 & 0 & \cdots & 0 & \cdots \\
%                \A_1^2 & \A_2^2 & 0 & \cdots & 0 & \cdots \\
%                0 & \A^3_2 & \A_3^3 & \cdots & 0 & \cdots \\
%                \vdots & \vdots & \vdots & \vdots & \vdots & \vdots \\
%              \end{array}
%            \right).
%\end{equation}
%with,
%and, $\mathcal{F}_0=(\F_0(t),0,0,\cdots)^\tra$.

\begin{problem}\label{prob1}
For the system (\ref{eq:qdc}), let $\F_0,\F_1,\F_2$ be time independent matrices and $s-$sparse (i.e. have at most $s$ nonzero entries in each row and column) and $\F_1$ be diagonalizable with eigenvalues $\lam_j$ of $\F_1$ satisfying $\textsf{Re}(\lam_n)\leq\cdots\leq \textsf{Re}(\lam_1)<0$. Let us define,
\begin{equation}\label{eq:R2}
R_2=\frac{1}{|\textsf{Re}(\lambda_1)|}(\|\xv_0\| \| \F_2\|+\frac{\|\F_0\|}{\|\xv_0\|}).
%R_2=\frac{\|\xv_0\|_2 \| \F_2\|_2}{|\textsf{Re}(\lambda_1)|}.
\end{equation}
Assume that we are given oracles $\Or_{\F_2}$, $\Or_{\F_1}$ and $\Or_{\F_0}$ that provide the nonzero entries of $\F_i,i=0,1,2$, respectively for any desired row or column. We are also given the value of $\|\xv_0\|_2$ and an oracle $\Or_{\xv}$ that maps $|00\cdots0\rangle\in \C^n$ to a quantum state proportional to $\xv_0$. The goal is to produce a quantum state proportional to the solution  $\xv(T)$ (of system (\ref{eq:qdc})), for a given $T>0$, within a specified $l_2$ error tolerance $\epsilon>0$.
\end{problem}

\begin{remark}
Under the condition (\ref{eq:R2}), system (\ref{eq:qdc}) can be rescaled $\xv\rightarrow \frac{\overline{\xv}}{\gamma }$ to (see Appendix in \cite{liu2021efficient} for details),
\begin{equation}\label{eq:qdct}
\dot{\overline{\xv}}=\overline{\F}_0+\overline{\F}_1\overline{\xv}+\overline{\F}_2\overline{\xv}^{[2]},
\end{equation}
such that the following relations hold,
\begin{eqnarray}
   \|\overline{\F}_2\|+\|\overline{\F}_0\|&<& |\textsf{Re}(\lam_1)|, \label{eq:condt1}\\
   \|\overline{\xv}(0)\| &<& 1,\label{eq:condt2}
\end{eqnarray}
where, $\overline{\F}_0=\gamma \F_0$, $\overline{\F}_1=\F_1$, $\overline{\F}_2=\frac{1}{\gamma}\F_2$  and $\gamma$ satisfies, %\textcolor{red}{I think that the F0 and F2 scaling is flipped}
\begin{equation}\label{eq:gamma}
\gamma=\frac{1}{\sqrt{\|\xv(0)\|r_{+}}}, \quad \text{with} \quad r_{\pm}=\frac{-\textsf{Re}(\lam_1)\pm\sqrt{(\textsf{Re}(\lam_1))^2-4\|\F_2\|\|\F_0\|}}{2\|\F_2\|}.
\end{equation}
Note that under this rescaling, the value of $R_2$ in (\ref{eq:R2}) remains unchanged.
\end{remark}

\begin{theorem}\label{them1}
[See \cite{liu2021efficient}] Consider a quadratic ODE instance as defined in Problem \ref{prob1}, with its CL as defined in (\ref{eq:carlfq}). Assume $R_2<1$, as defined in (\ref{eq:R2}), and let $q=\|\xv(0)\|/\|\xv(T)\|$. There exists a quantum algorithm that produces a state that approximates $\xv(T)/\|\xv(T)\|$ with error at most $\epsilon<1$, succeeding with probability $\Omega(1)$, and has query and gate complexity of order 
%(\textcolor{red}{both query and gate complexity are the same?}),
\begin{equation}\label{eq:com}
\frac{sT^2 q}{\epsilon}\cdot\textsf{poly}(\log T, \log \nx, \log 1/\epsilon).
\end{equation}
\end{theorem}
A more formal version of the above theorem can be found in the Appendix in \cite{liu2021efficient}. Note that the query complexity means the number of queries the quantum algorithm makes to the oracles $\Or_{\xv}$, $\Or_{\F_2}$, $\Or_{\F_1}$, and $\Or_{\F_0}$, while the gate complexity refers to the size of the quantum circuit required to implement the quantum algorithm. Also, it is assumed that oracle $\Or_{\xv}$ can map $|00\cdots0\rangle\in \C^n$ to a quantum state proportional to $\xv(0)$ in $\mathcal{O}(\text{poly}(\log \nx))$ time.

The quantum algorithm for simulating system (\ref{eq:qdc}), that achieves the complexity stated in Theorem \ref{them1}, is based on discretizing the truncated linear ODE (\ref{eq:carlfq}) using the forward Euler method, writing the discretized equations as a system on linear algebraic equations, and solving the linear system using QLSA. Following the nomenclature in~\cite{liu2021efficient}, we refer to this quantum algorithm as the quantum Carleman linearization (QCL) algorithm.

To construct the system of linear equations, the interval $[0, T]$ is divided into $\tsteps=\lceil T/h \rceil$ time steps, and the forward Euler method is applied to (\ref{eq:carlfq}) resulting in,
\begin{equation}\label{eq:disc}
\zvh^{j+1}=[\Id+\A_Nh]\zvh^{j}+h\bv,
\end{equation}
for $j\in[\tsteps+1]_0=\{0,1,\cdots,\tsteps\}$, where $\zvh^j=\zvh(h(j-1))$ and $\zvh^0=\zvh_{in}=(\xv(0)^\tra,(\xv(0)^{[2]})^\tra,\cdots,(\xv(0)^{[\nt]})^\tra)^\tra$. Letting $\zvh^j$ be equal for $j\in[\tsteps+p+1]_0 \setminus [\tsteps+1]_0$ for sufficiently large $p$, system (\ref{eq:disc}) can be expressed in the form of linear algebraic equations,
\begin{equation}\label{eq:lin}
\left(
  \begin{array}{ccccccc}
    \Id & 0 & 0 & \cdots & 0 & \cdots & 0 \\
    -[\Id+\A_N h] & \Id & 0 & \cdots & 0 & \cdots & 0 \\
    0 & 0 & \ddots & \ddots & \ddots & 0 & 0 \\
    0 & 0 &  -[\Id+\A_N h] & \Id & 0 & \cdots & 0 \\
    0 & 0 & 0 & \cdots & -\Id & \Id & 0 \\
    0 & 0 & 0 & 0 & \ddots & \ddots & 0 \\
    0 & 0 & 0 & \cdots & 0 & -\Id & \Id \\
  \end{array}
\right)\left(\begin{array}{c}
         \zvh^0 \\
         \zvh^1 \\
         \vdots \\
         \zvh^\tsteps \\
         \zvh^{\tsteps+1} \\
         \vdots \\
         \zvh^{\tsteps+p}
       \end{array}\right)
=\left(\begin{array}{c}
         \zvh_{in} \\
         h\bv \\
         \vdots \\
         h\bv \\
         0 \\
         \vdots \\
         0
       \end{array}\right).
\end{equation}
In the above system, the first $\nx$ components of $\zvh^j$ for $j\in[\tsteps+p+1]_0\setminus[\tsteps]_0$, i.e. $\zvh_1^j$,  approximates the exact solution $\xv(T)$ of (\ref{eq:qdc}) up to a normalization constant. Using the high-precision variant of the QLSA \cite{childs2017quantum} to solve the system (\ref{eq:lin}), and post-selecting on $j$ to produce $\zvh_1^j/\|\zvh_1^j\|$, for some $j\in[\tsteps+p+1]_0\setminus[\tsteps]_0$, the resulting error is at most,
\begin{equation}\label{eq:error}
\epsilon=\max_{j\in [\tsteps+p+1]_0\setminus[\tsteps]_0}\Bigg \| \frac{\xv(T)}{\|\xv(T)\|}-\frac{\zvh_1^j}{\|\zvh_1^j\|}\Bigg \|.
\end{equation}
This error includes contributions from the CL, forward Euler method, and QLSA steps of the algorithm.  The complexity analysis result in Theorem \ref{them1} uses an improved convergence error analysis for CL, a bound on the global error of the Euler method, an upper bound on the condition number of the linear algebraic equations (\ref{eq:lin}), and a lower bound on the success probability of the final measurement. Specifically, without loss of generality, the system (\ref{eq:qdc}) is first transformed into form (\ref{eq:qdct}) such that conditions (\ref{eq:condt1}-\ref{eq:condt2}) are satisfied. Then, for the desired accuracy $\epsilon$, $\nt$ (the truncation level for Carleman) and $h$ (time step in Euler method) are selected based on the CL and Euler method's error analysis, respectively, such that $\Bigg \| \frac{\xv(T)}{\|\xv(T)\|}-\frac{\zvh_1^m}{\|\zvh_1^m\|}\Bigg \|\leq \epsilon$. These values of $\nt$ and $h$ are then used to estimate the condition number and sparsity of the linear system (\ref{eq:lin}). The query/gate complexity of QLSA for specifically solving the system (\ref{eq:lin}) are determined based on the complexity result derived in \cite{childs2017quantum}. Finally, the QLSA computational complexity is combined with the complexity for (initial) state preparation, and the number of iterations of the QCL algorithm required for achieving desired success probability of the final measurement during post selection, leading to Theorem \ref{them1}. We next summarize the state preparation and post selection process, see \cite{liu2021efficient} for further details.

The linear system (\ref{eq:lin}) can be expressed in quantum form as,
\begin{equation}\label{eq:Qlin}
L|Z\rangle=|B\rangle,
\end{equation}
where,
\begin{equation}\label{eq:Ldef}
L=\sum_{j=0}^{m+p}|j\rangle \langle j|\otimes \Id-\sum_{j=1}^{m}|j\rangle\langle j-1|\otimes [\Id+\A_N h]-\sum_{j=m+1}^{m+p}|j\rangle\langle j-1|\otimes \Id, \notag
\end{equation}
%the solution $|Z\rangle$ is
\begin{equation}\label{eq:Zdef}
|Z\rangle=\sum_{j=1}^{m+p}\sum_{i=1}^{\nt}|\zvh_i^j\rangle|i\rangle|j\rangle,\text{ and } |B\rangle=\frac{1}{\sqrt{B_m}}\left(\|\zvh_{in}\| |0\rangle\otimes|\zvh_{in}\rangle |+\sum_{j=1}^m \|h\bv \||j\rangle\otimes |h\bv\rangle\right),\notag
\end{equation}
%and
%\begin{equation}\label{eq:Bdef}
%|B\rangle=\frac{1}{\sqrt{B_m}}\left(\|\zvh_{in}\| |0\rangle\otimes|\zvh_{in}\rangle |+\sum_{j=1}^m \|h\bv \||j\rangle\otimes |h\bv\rangle\right),
%end{equation}
and $B_m=\|\zvh_{in}\|^2+mh^2\|\bv\|^2$ is the normalization constant. We have used the standard ``bra" $|\cdot\rangle$ and ``ket" $\langle \cdot |$ notation to represent a quantum state and its conjugate transpose, respectively \cite{nielsen2002quantum}.  The entries of $|B\rangle$ are composed on $|\zvh_{in} \rangle$ and $|\bv h\rangle$, which can be prepared in superposition by mapping $|0\rangle\rightarrow |\zvh_{in} \rangle $ and $|0\rangle\rightarrow |\bv h\rangle$, respectively using a standard procedure. Overall, this process takes $\mathcal{O}(\nt)$ queries to $\Or_{\xv}$ and $\mathcal{O}(m)$ queries to $\Or_{\F_0}$ to prepare $|B\rangle$. Since the initial state vector $|\zvh_{in} \rangle$ is a direct sum over spaces of different dimensions, it can be cumbersome to deal with directly. Instead, one can prepare an equivalent state of the form,
\begin{equation}
\wvh_{in}=\left((\zvh_{in}\otimes \vb_0^{\nt-1})^\tra, (\zvh_{in}^{[2]}\otimes\vb_0^{\nt-2})^\tra, \cdots, (\zvh_{in}^{[\nt]})^\tra\right)^\tra, \notag
\end{equation}
which has a convenient tensor product structure. Here, $\vb_0$ is a standard vector such as, $\vb_0=|0\rangle$. Using standard techniques, all the operations that one would otherwise apply to $\zvh_{in}$ can be applied to $\wvh_{in}$, with the same effect.

For post-selection, the solution is decomposed into the form,
\begin{equation}\label{eq:Zdefdec}
|Z\rangle=|Z_b\rangle+|Z_g\rangle,
\end{equation}
where,
\begin{eqnarray}
 |Z_b\rangle  &=& \sum_{j=0}^{m-1}\sum_{i=1}^{\nt}|\zvh_i^j\rangle|i\rangle|j\rangle+\sum_{j=m}^{m+p}\sum_{i=2}^{\nt}|\zvh_i^j\rangle|i\rangle|j\rangle, \notag \\
 |Z_g\rangle  &=& \sum_{j=m}^{m+p}|\zvh_1^j\rangle|1\rangle|j\rangle.\notag
\end{eqnarray}
Note that, $\zvh_1^j=\zvh_1^m$ for all $j\in\{m,m+1,\cdots,m+p\}$. Then the probability $P_{\text{measure}}$ of measuring  $|Z_g\rangle$, i.e. solution of interest, can be lower bounded as
\begin{equation}\label{eq:prob}
P_{\text{measure}}=\frac{\||Z_g\rangle \|^2}{\||Z\rangle\|^2}\geq \frac{p+1}{9(m+p+1)\nt q^2}.
\end{equation}
Choosing $m=p$, we get $P_{\text{measure}}=\mathcal{O}(1/\nt q^2)$. Using amplitude amplification, $\mathcal{O}(\sqrt{\nt}q)$ iterations suffice to succeed with a constant probability.

\begin{remark}
The authors in \cite{liu2021efficient} also provide a lower bound on the worst-case complexity of quantum algorithms for general quadratic ODEs, showing that the problem is intractable for $R_2\geq \sqrt{2}$.
\end{remark}

\subsection{Generalization to $k$-th Order ODE}\label{sec:QCgenODE}

\begin{problem}\label{prob2}
Consider system (\ref{eq:gen}), and let $\F_i, i=0,\dots,k$ be time independent $s-$sparse matrices and $\xv_0\in \R^\nx$ be the initial condition.  Let $\tF_0,\tF_1$ and $\tF_2$ be as defined in (\ref{eq:tF1}) and (\ref{eq:tF2}), respectively. Assume $\tF_1$ is diagonalizable and its eigenvalues satisfy $\tilde{\lam}_{n_k}\leq\cdots\leq\tilde{\lam}_1<0$.  Define,
\begin{equation}\label{eq:Rk}
%R=\frac{1}{|\textsf{Re}(\lambda_1)|}(\|\xv_0\| \| \F_2\|+\frac{\|\F_0\|}{\|\xv_0\|}).
R_k=\frac{1}{|\textsf{Re}(\tilde{\lambda}_1)|}\left((k-1)\sum_{j=2}^k \|\F_j\|\sqrt{\sum_{i=1}^{k-1}\|\xv_0\|^{2i}}+\frac{\|\F_0\|}{\sqrt{\sum_{i=1}^{k-1}\|\xv_0\|^{2i}}}\right).
%R_k=\frac{\|\txv_0\|_2 \| \tF_2\|_2}{|\textsf{Re}(\lambda_1)|}.
\end{equation}
Assume that we are given oracles $\Or_{\tF_i}$ that provide nonzero entries of $\tF_i,i=0,1,2$ respectively for any desired row or column. We are also given the value of $\|\xv_0\|$ and the an oracle $\Or_x$ that maps $|00\cdots0\rangle\in \C^n$ to a quantum state proportional to $\tilde{\xv}_0=(\xv_0^\tra,(\xv^{[2]}_0)^\tra,\dots,(\xv_0^{[k-1]})^\tra)^\tra$. The goal is to produce a quantum state proportional to the solution  $\xv(T)$ of the system (\ref{eq:gen}), for some given $T>0$, within specified $l_2$ error tolerance $\epsilon>0$.
\end{problem}

\begin{problem}\label{prob3}
Consider the homogeneous version of Problem \ref{prob2}, i.e. $\F_0\equiv0$, and assume that $\F_1$ is diagonalizable with its eigenvalues $\lam_j$ satisfying $\textsf{Re}(\lam_n)\leq\cdots\leq \textsf{Re}(\lam_1)<0$. Without loss of generality, we assume that $\F_1$ is diagonal with diagonal entries as its eigenvalues.  Define,
\begin{equation}\label{eq:Rk0}
R^0_k=\frac{(k-1)\sqrt{\sum_{i=1}^{k-1}\|\xv_0\|^{2i}}\sum_{i=2}^k \|\F_i\|}{|\textsf{Re}(\lambda_1)|}.
\end{equation}
\end{problem}

\begin{lemma}\label{lem2}
Let $\F_i,i=0,\dots,k$ be $s-$sparse as assumed in Problem \ref{prob2}. Let $\tF_0$, $\tF_1$ and $\tF_2$ be as defined in (\ref{eq:tF1}) and (\ref{eq:tF2}), respectively. Then the sparsity of $\tF_0$ is $s$,  and that of $\tF_1$ and $\tF_2$ is at most $s_k\leq k(k-1)s/2$.
\end{lemma}
\begin{proof}
From the definition of $\tF_0$, it is clear that the sparsity of $\tF_0$ is the same as the sparsity of $\F_0$. It is clear that the sparsity of a matrix does not change with pre- and post-Kronecker products with an identity matrix, and the sparsity of addition of two sparse matrices is in the worst-case additive. Thus, sparsity of matrices $\A^{i}_{i+j-1}$ defined in (\ref{eq:Adef}) is at most $i\times s$. Consequently, sparsity of $\tF_1$ and $\tF_2$  in worst case is $\sum_{i=1}^{k-1}is=sk(k-1)/2$.
\end{proof}

\begin{lemma}\label{lem1}
Let $\tilde{\lam}_i,i=1,\dots,\nx_k$ be eigenvalues of $\tF_1$ as defined in (\ref{eq:tF1}). Then under the assumptions of Problem \ref{prob3}, i.e.,  $\F_0\equiv0$ and $\F_1$ is diagonalizable with its eigenvalues $\lam_j$ satisfying $\textsf{Re}(\lam_n)\leq\cdots\leq \textsf{Re}(\lam_1)<0$,  the largest eigenvalue of $\tF_1$ satisfies
\begin{equation}\label{eq:larlam}
\max_{i}\textsf{Re}(\tilde{\lam}_i)=\textsf{Re}(\lam_1).
\end{equation}
%where, $\lam_1$ is the largest eigenvalue of $\F_1$.
\end{lemma}
\begin{proof}
When $\F_0\equiv 0$, $\A^{i}_{i-1}\equiv0,i=1,\dots,k-1$ and $\tF_1$ reduces to an upper block diagonal form. Therefore, the eigenvalues of $\tF_1$ are formed from eigenvalues of $\A_i^i,i=1,\dots,k-1$. Substituting $j=1$  in the relation (\ref{eq:req}) leads to the following recurrence relation,
\begin{equation}\label{eq:reqd}
\A^{i}_{i}=\A^{i-1}_{i-1}\otimes \Id_{\nx\times\nx}+\Id_{\nx\times\nx}^{[i-1]}\otimes \A^1_1,
\end{equation}
where, $\A_1^1=\F_1$. As $\F_1$ is diagonal, it then follows from above relation that $\A_i^i,i=1,\dots,k-1$ are diagonal with eigenvalues of the form $\sum_{j=1}^{i}\alpha_j\lam_j$ with $\alpha_j\in\N$. Since, $\textsf{Re}(\lam_n)\leq\cdots\leq \textsf{Re}(\lam_1)<0$, it follows that $\max_{i}\textsf{Re}(\tilde{\lam}_i)=\textsf{Re}(\lam_1)$.
\end{proof}

\begin{theorem}\label{them2}
Consider an instance of the $k$-the order inhomogeneous polynomial ODE as defined in Problem \ref{prob2}. Assume $R_k<1$, as defined in (\ref{eq:Rk}), and let,
\begin{equation}\label{eq:qksk}
q_k=\frac{\sqrt{\sum_{i=1}^{k-1}\|\xv_0\|^{2i}}}{\sqrt{\sum_{i=1}^{k-1}\|\xv(T)\|^{2i}}}, \quad s_k=sk(k-1)/2.
\end{equation}
%q_k=\frac{\sqrt{\sum_{i=1}^{k-1}\|\xv_0\|^{2i}}}{\sqrt{\sum_{i=1}^{k-1}\|\xv(T)\|^{2i}}},
Then the QCL algorithm applied to the quadratic transformation (\ref{eq:qd}) produces a state that approximates $\xv(T)/\|\xv(T)\|$ with an error at most $\epsilon<1$, succeeding with probability $\Omega(1)$, and with query and gate complexity  of order,
\begin{equation}\label{eq:comg}
\frac{s_kT^2 q_k}{\epsilon}\textsf{poly}(\log T, (k-1)\log \nx, \log 1/\epsilon).
\end{equation}
\end{theorem}

\begin{proof}
Consider the quadratic transformation (\ref{eq:qd}) of the $k$-th order polynomial system (\ref{eq:gen}) as described in Section \ref{sec:conv}. Note that $\nx_k=\textsf{dim } \txv=\frac{\nx^k-\nx}{\nx-1}$, which is $\mathcal{O}(\nx^{k-1})$ for large $\nx$. Let $\tilde{\lam}_i, i=1,\dots,\nx_k$ denote the eigenvalues of $\tF_1$. From Lemma \ref{lem0} and  $\|\tF_0\|=\|\F_0\|$, it follows that $\tilde{R}_2$ as defined in (\ref{eq:R2}) for system (\ref{eq:qd}) satisfies,
\begin{equation}\label{eq:R2g}
\tilde{R}_2=\frac{1}{|\textsf{Re}(\tilde{\lambda}_1)|}\left(\|\txv_0\| \| \tF_2\|+\frac{\|\tF_0\|}{\|\txv_0\|}\right)\leq R_k<1,
\end{equation}
and $\tilde{q}=\|\txv(0)\|/\|\txv(T)\|=\frac{\sqrt{\sum_{i=1}^{k-1}\|\xv(0)\|^{2i}}}{\sqrt{\sum_{i=1}^{k-1}\|\xv(T)\|^{2i}}}$. Furthermore, from Lemma \ref{lem2}, the sparsity of $\tF_0$ is $s$  and $\tF_1$ and $\tF_2$ both have sparsity $s_k=k(k-1)s/2$. Theorem \ref{them2} follows by applying Theorem \ref{them1} to system (\ref{eq:qd}), and by noting that the QCL algorithm produces a state $\frac{\hat{\tilde{\zv}}_1^m}{\|\hat{\tilde{\zv}}_1^m\|}$ such that $\Bigg \| \frac{\txv(T)}{\|\txv(T)\|}-\frac{\hat{\tilde{\zv}}_1^m}{\|\hat{\tilde{\zv}}_1^m\|}\Bigg \|\leq \epsilon$. This result implies $\Bigg \| \frac{\xv(T)}{\|\txv(T)\|}-\frac{\hat{\tilde{\zv}}_{1n}^m}{\|\hat{\tilde{\zv}}_1^m\|}\Bigg \|\leq \epsilon$, where $\hat{\tilde{\zv}}_{1\nx}^m$ are the first $\nx$ components of $\hat{\tilde{\zv}}_{1}^m$, and consequently,
\begin{equation}\label{eq:err}
\Bigg \| \frac{\hat{\tilde{\zv}}_{1\nx}^m}{\|\hat{\tilde{\zv}}_{1\nx}^m\|}-\frac{\xv(T)}{\|\xv(T)\|}\Bigg \|=\Bigg \|  \frac{\hat{\tilde{\zv}}_{1\nx}^m}{\|\hat{\tilde{\zv}}_{1\nx}^m\|}-\frac{\xv(T)}{\|\txv(T)\|}\frac{\|\txv(T)\|}{\|\xv(T)\|}    \Bigg \|\leq\Bigg \|    \frac{\hat{\tilde{\zv}}_{1\nx}^m}{\|\hat{\tilde{\zv}}_{1\nx}^m\|}-\frac{\xv(T)}{\|\txv(T)\|}\Bigg \|\leq  \epsilon.
\end{equation}
Note that we used the fact that $\|\xv\|\leq \|\txv\|$ to get the above result. Thus, by extracting the first $\nx$ components of $\frac{\hat{\tilde{\zv}}_1^m}{\|\hat{\tilde{\zv}}_1^m\|}$ produced by QCL, one obtains a state that approximates $\xv(T)/\|\xv(T)\|$ with error of at most $\epsilon$.
%\begin{equation}\label{eq:err}
%\Bigg \| \frac{\hat{\tilde{\zv}}_{1\nx}^m}{\|\hat{\tilde{\zv}}_{1\nx}^m\|-\frac{\xv(T)}{\|\txv(T)\|}}\Bigg \|\leq \Bigg \| \frac{\xv(T)}{\|\txv(T)\|}-\frac{\hat{\tilde{\zv}}_{1\nx}^m}{\|\hat{\tilde{\zv}}_1^m\|}\frac{\|\hat{\tilde{\zv}}_1^m\|}{\|\hat{\tilde{\zv}}_{1\nx}^m\|}\Bigg \|\leq\Bigg \| \frac{\xv(T)}{\|\txv(T)\|}-\frac{\hat{\tilde{\zv}}_{1n}^m}{\|\hat{\tilde{\zv}}_1^m\|}\Bigg \|\leq  \epsilon
%\end{equation}
%where, $\hat{\tilde{\zv}}_{1\nx}^m$ are first $\nx$ components of $\hat{\tilde{\zv}}_{1}^m$ and we used the fact that $\|\hat{\tilde{\zv}}_{1\nx}^m\|\leq \|\hat{\tilde{\zv}}_{1}^m\|$. Thus, by extracting first $\nx$ components of $\frac{\hat{\tilde{\zv}}_1^m}{\|\hat{\tilde{\zv}}_1^m\|}$ produced by QCL and renormalizing one obtains a state that approximates $\xv(T)/\sqrt{\sum_{i=1}^{k-1}\|\xv(T)\|^{2i}}$ with an error at most $\epsilon$.
\end{proof}
%one can obtain $\xv(T)/\|\xv(T)\|$ by extracting first $\nx$ entries from  $\txv(T)/\|\txv(T)\|$ and renormalizing.
%\begin{equation}\label{eq:R2g}
%\tilde{R}_2=\frac{\|\txv_0\|_2 \| \tF_2\|_2}{|\textsf{Re}(\tilde{\lam}_1)|}=\frac{\|\txv_0\|_2 \| \tF_2\|_2}{|\textsf{Re}(\lam_1)|}=R_k<1.
%\end{equation}
%where, we have used the fact that $\|\txv\|^2=\sum_{i=1}^{k-1}\|\xv\|^{2i}$. Further by same argument,
%\begin{equation}\label{eq:R2g}
%R_2=\frac{\|\txv_0\| \| \tF_2\|}{|\textsf{Re}(\lam_1)|}\leq \frac{(k-1)\sqrt{\sum_{i=1}^{k-1}\|\xv_0\|^{2i}}\sum_{j=2}^k \|\F_j\|}{|\textsf{Re}(\lambda_1)|}=R_k<1.
%\end{equation}
%\frac{\sqrt{\sum_{i=1}^{k-1}\|\xv_0\|^{2i}}}{\sqrt{\sum_{i=1}^{k-1}\|\xv(T)\|^{2i}}}, \xv(T)/\sqrt{\sum_{i=1}^{k-1}\|\xv(T)\|^{2i}}
%Also by using $\|\cdot\|_\infty$ norm, one can bound $R_k$ in terms of parameters of original system (\ref{eq:gen}), as follows
%\begin{equation}\label{eq:R2ginf}
%R_k=\frac{\|\txv_0\|_2 \| \tF_2\|_2}{|\textsf{Re}(\lam_1)|}\leq \frac{\nx^k-\nx}{\nx-1}\frac{\|\txv_0\|_\infty \| \tF_2\|_\infty}{|\textsf{Re}(\lam_1)|}\leq %\frac{\nx^k-\nx}{\nx-1}\frac{\max_{i=1,\cdots,k-1}\{\|\xv_0\|_\infty^{i}\}\sum_{j=2}^k \|\F_j\|_\infty}{|\textsf{Re}(\lambda_1)|},
%\end{equation}
%where, we have used the inequality (\ref{eq:nrF2}) and the fact $\|\txv\|_\infty=\max\{\|\xv\|_\infty^{i}: i=1,\cdots,k-1\}$ which follows from (\ref{eq:normKron}).

\begin{corollary}\label{them3}
Consider an instance of the $k$-the order homogeneous polynomial ODE as defined in Problem \ref{prob3}. Assume $R^0_k<1$ as defined in  (\ref{eq:Rk0}). Then Theorem \ref{them2} holds for Problem \ref{prob3} with $R_k\leftarrow R_k^0$, and $q_k$ and $s_k$ as defined in (\ref{eq:qksk}).
\end{corollary}
\begin{proof}
Under the assumptions of Problem \ref{prob3}, $R_k$ reduces to,
\begin{equation}\label{eq:rkr0}
R_k=\frac{(k-1)\sqrt{\sum_{i=1}^{k-1}\|\xv_0\|^{2i}}\sum_{i=2}^k \|\F_i\|}{|\textsf{Re}(\tilde{\lambda}_1)|}=R_k^0,
\end{equation}
as $\textsf{Re}(\tilde{\lambda}_1)=\textsf{Re}(\lambda_1)$ (see Lemma~\ref{lem1}). Hence, the corollary follows from the Theorem \ref{them2}.
\end{proof}

Some remarks follow.

\begin{remark}
For $k=2$, Theorem \ref{them2} reduces to Theorem \ref{them1}.
\end{remark}

\begin{remark}
The quantum algorithm for simulating a $k$-th order polynomial  system (\ref{eq:gen}) involves transforming (\ref{eq:gen}) into a quadratic form (\ref{eq:qd}),  and then invoking the QCL algorithm as outlined in Section \ref{sec:QCquadODE}. We show in the Appendix (Lemma \ref{appexLem}), that the CL of the quadratic form (\ref{eq:qd}) is equivalent to the CL of the $k$-th order polynomial  system  (\ref{eq:gen}). Thus, alternatively, one can directly use the truncated CL (\ref{eq:carlf}) with truncation level $\nt^\prime=\nt(k-1)$. This truncation is equivalent to the chosen truncation level $\nt$ for CL of the quadratic form (\ref{eq:qd}). One can then apply the steps in QCL to the truncated CL of the direct form. Note that this approach is more efficient than the CL of the quadratic form since the quadratic structure (\ref{eq:qd}) introduces redundant equations that can be avoided by working directly with CL of (\ref{eq:gen}).
\end{remark}

\begin{remark}
Note that when $\F_0$ is time-dependent, $\tF_1$ also becomes time-dependent, leading to a time-dependent quadratic system (\ref{eq:qd}). Theorem \ref{them1} cannot be applied in that case and, consequently, Theorem \ref{them2} cannot be generalized for time-dependent $\F_0$ using the transformation approach described above.
\end{remark}

\section{Numerical Example}\label{sec:num}
To numerically demonstrate the proposed approach, we consider the reaction-diffusion PDE. Reaction-diffusion systems arise in many domains including chemistry, biology, geology, physics,  and ecology \cite{volpert2014elliptic}. The simplest reaction-diffusion equation is the 1-D PDE of the form,
\begin{equation}\label{eq:rd}
\partial_t u=\kappa\partial_{yy}u+R(u),
\end{equation}
on a domain $y\in [0,L]$, with initial condition $u(y,0)=u_0(y)$, and boundary conditions $\partial_y u(t,0)=g_0(t), \partial_y u(t,L)=g_L(t), \forall t\geq 0$.
The choice $R(u)=u(1-u)$ yields Fisher's equation that was originally used to describe the spreading of biological populations; the choice $R(u)=u(1-u^2)$ leads to the Newell–Whitehead-Segel equation used to describe Rayleigh–Benard convection; while the choice $R(u)=u^2-u^3$ results in the Zeldovich equation  used in combustion applications, see \cite{gilding2004travelling} for details.
%\cite{newell1969finite,segel1969distant}

For our demonstration we choose,
\begin{equation}\label{eq:allee}
R(u)=u(1-u)(u-a),
\end{equation}
where $ a\in [0,1/2]$, and captures the Allee effect \cite{bose2017allee}.  The Allee effect has to do with the fact that the fitness of small populations is sometimes negative, i.e., if the population density is too small, the species or group of individuals will not survive. To model this effect, an additional $u-a$ factor is added to the reaction term $R(u)=u(1-u)$ of the Fisher's equation. Furthermore, we consider a zero flux Neumann boundary condition, i.e. $g_0(t)=g_L(t)=0, \forall t\geq 0$, and an initial condition of the form,
\begin{equation}\label{eq:init}
u_0(y)= \begin{cases}
  u_{in}  & \text{ if } 0<y<y^* \\
  0 & \text{ otherwise}
\end{cases},
\end{equation}
where, $u_{in},y^* \in \R$ are constants. We discretize the PDE (\ref{eq:rd}) in space at $\nx$ discrete spatial locations $y_i=(i-1)\Delta y, i=1,\cdots,\nx$, such that $\Delta y=L/(\nx-1)$, and let $\uv(t)=(u(y_1,t),\cdots,u(y_\nx,t))^\tra$.  Then $\uv$  satisfies the ODE,
\begin{equation}\label{eq:rdd}
\dot{\uv}=\F_1\uv+\F_2\uv^{[2]}+\F_3\uv^{[3]},
\end{equation}
with the initial condition $\uv(0)=(u_0(y_1),\cdots,u_0(y_\nx))^\tra$ and
\begin{equation}\label{eq:rdF1}
\F_1=\left(
       \begin{array}{cccccc}
         \alpha_0 & \beta & 0 & \cdots & 0 & 0 \\
         \beta & \alpha & \beta & \cdots & 0 & 0 \\
         \vdots &  \vdots &  \vdots &  \vdots &  \vdots &  \vdots \\
         0 & 0 & 0 & \cdots & \beta & \alpha_0 \\
       \end{array}
     \right),
\end{equation}
%\ag{It's not clear what the pattern of entries on the diagonal is since (2,2) entry is $\alpha$ and the last entry is $\alpha_0$}
where, $\beta=\kappa/(\Delta y)^2$, $\alpha_0=a-\beta$,  and $\alpha=a-2\beta$, 
\begin{equation}\label{eq:rdF2}
\F_2(i,j)= \begin{cases}
  1+a  & \text{ if } j=(i-1)\nx+i \\
  0 & \text{ otherwise},
\end{cases}
\end{equation}
and
\begin{equation}\label{eq:rdF3}
\F_3(i,j)= \begin{cases}
  -1  & \text{ if } j=(i-1)\nx^2+(i-1)\nx+i \\
  0 & \text{ otherwise}.
\end{cases}
\end{equation}
Note that the first and last diagonal entries in $\F_1$~(\ref{eq:rdF1}) are different from the other diagonal entries due to the zero flux Neumann boundary condition. The equation (\ref{eq:rdd}) is of the form (\ref{eq:gen}) with $k=3$ and $\F_0\equiv 0$, and, consequently, we generate the Carleman system (\ref{eq:carl}) using the steps outlined previously. The resulting Carleman linear ODE is then integrated using the forward Euler method.

\begin{figure}[hbt!]
\centerline{\includegraphics[scale=0.55]{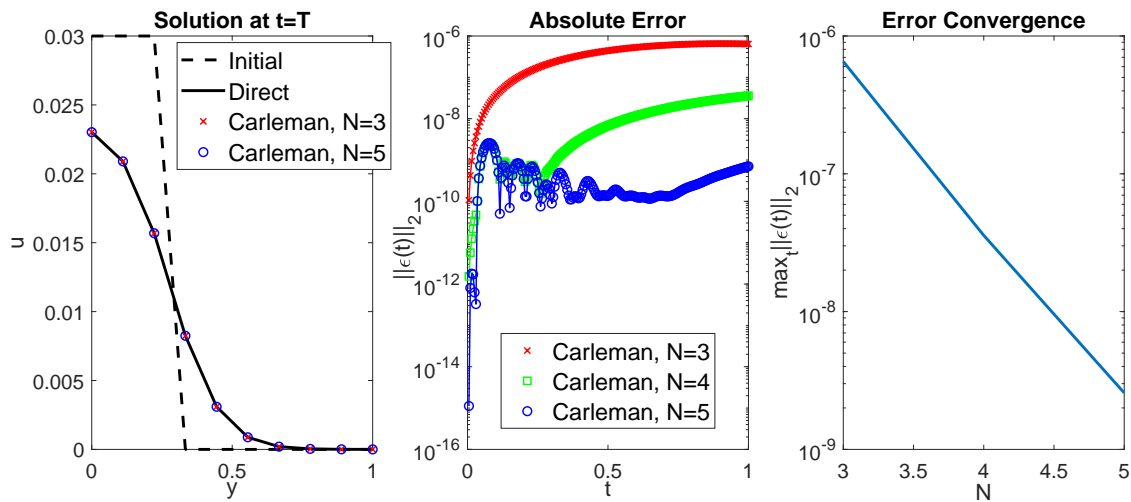}}
\caption{Results for $L=1,a=1/4,\kappa=(\Delta y)^2,y^*=0.3,u_{in}=0.03,\text{and}\,\,T=1$ which leads to $R^0_k<1$. (Left) Initial condition and solution plotted at $t=T$. (Center) $l_2$ norm of the absolute error between the Carleman solutions at various truncation levels $\nt$ and  the direct solution obtained via direct integration of (\ref{eq:rdd}). (Right) Convergence of the time-maximum error as a function of truncation level $\nt$.}
\label{fig:results1}
\end{figure}

\begin{figure}[hbt!]
\centerline{\includegraphics[scale=0.55]{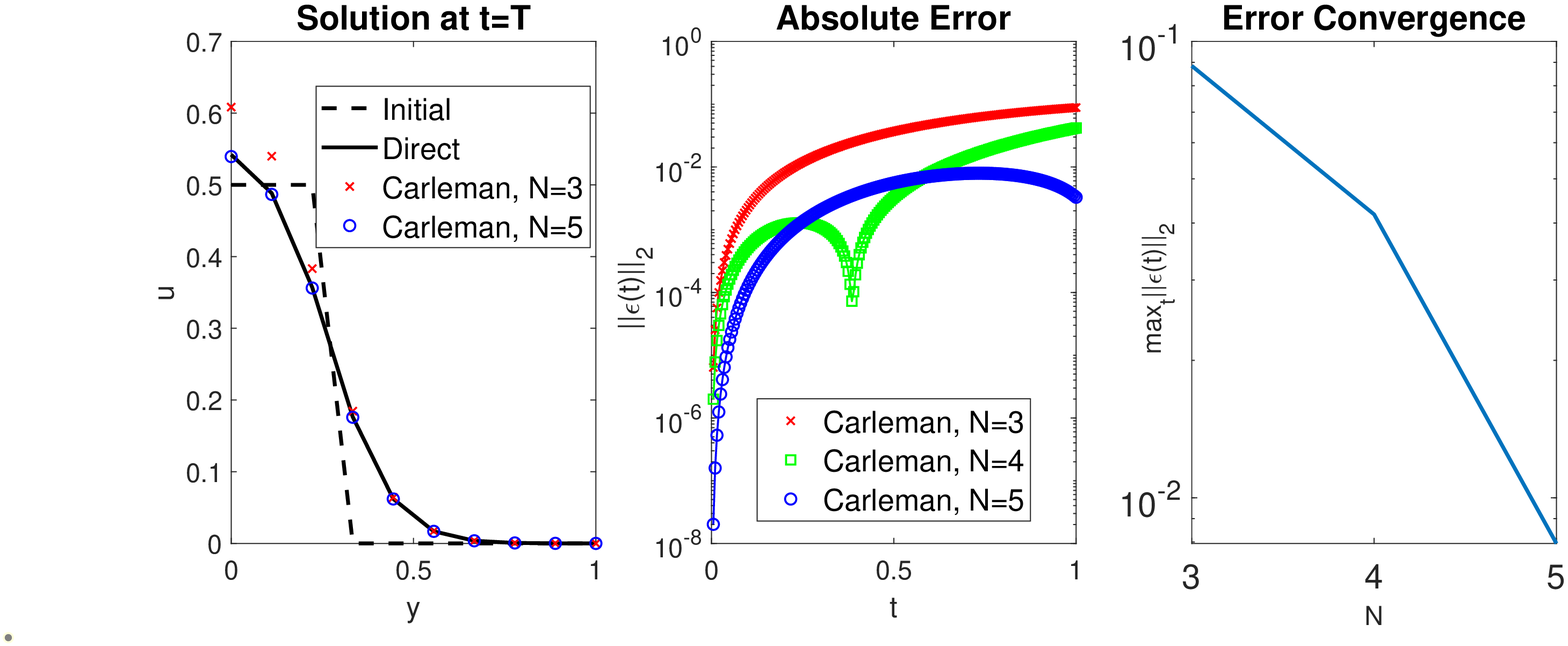}}
\caption{Similar as Figure \ref{fig:results1}, but for $u_{in}=0.5$ which results in $R^0_k>1$.}
\label{fig:results2}
\end{figure}

For the numerical demonstration, we use parameters: $L=1,a=1/4,\kappa=(\Delta y)^2,y^*=0.3,u_{in}=0.03, \text{and}\,\, T=1$ , and found that spatial discretization with $\nx = 10$ was adequate to resolve the solution. Figure \ref{fig:results1} compares the Carleman solution for different levels of truncation $\nt$ with the direct solution obtained via direct integration of (\ref{eq:rdd}). By inserting matrices $\F_i,i=0,\hdots,3$  from equation  (\ref{eq:rdd}) into the definition of $R^0_k$ for the Neumann boundary conditions, we find that $\textsf{Re}(\lambda_1)<0$ and $R^0_k=0.94<1$, as required.
We can see from the absolute error plot in Figure \ref{fig:results1}, that, as expected, the maximum absolute error as a function of time decreases exponentially with increasing truncation levels $\nt$ (in this example, the maximum Carleman truncation level is $\nt = 5$). 

Next, we consider the case for $u_{in}=0.3$, while all other parameters remain the same. Here, we find that $\textsf{Re}(\lambda_1)<0$, but $R^0_k=20.62>1$. Even though this does not satisfy the requirement $R^0_k< 1$, we see from  Figure \ref{fig:results2}, that the maximum absolute error
over time decreases exponentially as the truncation level $\nt$ is incremented.  This suggests that, in this example, the error of the classical Carleman method converges exponentially with $\nt$, even though $R^0_k>1$. Similar convergence phenomena was also observed in the Burgers PDE example considered in~\cite{liu2021efficient}. It is likely that the $R^0_k<1$ condition is, perhaps, too conservative is most cases. We note that the error in this case is much larger compared to the previous case for the same levels of truncation (compare Figures~\ref{fig:results1} and~\ref{fig:results2}).

%For numerical demonstration we use parameters: $L=1,a=1/4,\kappa=(\Delta y)^2,y^*=0.3,u_{in}=0.5,T=1$ , and found that $\nx = 10$ spatial discretization points which were sufficient to resolve the solution. Figure \ref{fig:results} compares the Carleman solution for different levels of truncation $\nt$ with the direct solution obtained via direct integration of (\ref{eq:rdd}). By inserting the matrices $\F_i,i=0,\dots,3$  corresponding to equation  (\ref{eq:rdd}) into the definition of $R^0_k$ (\ref{eq:Rk0}), we find that $\textsf{Re}(\lambda_1)<0$ is indeed negative as required, given Neumann boundary conditions, but the parameters used in this example result in $R^0_k=20.62$. Even though this does not satisfy the requirement $R^0_k< 1$, we see from the absolute error plot in Figure \ref{fig:results} that the maximum absolute error
%over time decreases exponentially as the truncation level $\nt$ is incremented (in this example, the maximum Carleman truncation level considered is $\nt = 5$). This suggests that in this example, the error of the classical Carleman method converges exponentially with $\nt$, even though
%$R^0_k>1$, which was also observed in the Burgers PDE example considered in~\cite{liu2021efficient}.

\section{Conclusions}\label{sec:conc}
The emergence of quantum computers presents a novel opportunity for addressing several computational bottlenecks and challenges that exist in classically simulating dynamical systems that are represented by nonlinear ordinary or partial differential equations. However, to extend the range of applicability of quantum computers to these settings, there is an urgent need to develop novel quantum algorithms. One promising line of work for simulating dynamical systems on quantum computers has focused on exploiting the quantum linear systems algorithm (QLSA) for exponential speed-up. In particular, in previous work~\cite{liu2021efficient}, the authors have used Carleman linearization along with the forward Euler numerical scheme to convert quadratic ODEs into a linear system of algebraic equations. The resulting linear system was constructed in a manner so as to be amenable to the QLSA framework. However, given that was only applicable to quadratic vector fields, it was restricted in its range of applicability.

In this paper, we have generalized the Carleman linearization based framework for simulating quadratic polynomial ODEs on quantum computers~\cite{liu2021efficient} to nonlinear ODEs with $k$-th degree polynomial vector fields for arbitrary (finite) values of $k$. Our approach is based on either mapping the $k$-th degree polynomial ODE to a higher dimensional quadratic polynomial ODE and applying the Carleman linearization or using the Carleman framework to directly cast the $k$-th degree polynomial ODE into a corresponding linear form. We then generalized the existing complexity analysis in~\cite{liu2021efficient} and proved an explicit polynomial computational scaling on the degree $k$ of the polynomial vector field. In future work, we plan to exploit the novel Carleman convergence analysis established in \cite{amini2022carleman}, to extend the our framework for the setting of time-dependent $k$-th degree polynomial vector fields and bound the truncation error as $k\rightarrow \infty$. Also, we plan to explore extensions for simulating arbitrary (non-polynomial) dynamical systems on quantum computers. Although, in general, one can approximate such dynamical systems by polynomial vector fields, these approximations may not hold for structurally unstable dynamical systems.

\section*{Appendix}\label{appex}
\begin{lemma}\label{appexLem}
The CL of the transformed quadratic form (\ref{eq:qd}) with truncation level $\nt$ is equivalent to the CL  of $k$-th order polynomial  system  (\ref{eq:gen}) with truncation level $\nt^\prime=\nt(k-1)$.
\end{lemma}

\begin{proof}
First note that the system (\ref{eq:qd}) can be expressed as,
\begin{eqnarray}
  \dot{\txv}_1 &=& \A_1^1\txv_1+\A^1_2\txv_2+\cdots+\A^1_{k-1}\txv_{k-1}+\A^1_k\txv_1\otimes \txv_{k-1}+\A^1_0, \notag\\
  \dot{\txv}_2 &=& \A_1^2\txv_1+\A^2_2\txv_2+\cdots+\A^2_{k-1}\txv_{k-1}+\A^2_k\txv_1\otimes \txv_{k-1}+\A^2_{k+1}\txv_2\otimes \txv_{k-1}, \notag \\
   &\vdots & \label{eq:sysCL1}\\
  \dot{\txv}_{k-1} &=& \A^{k-1}_{k-2}\txv_{k-2}+\A^{k-1}_{k-1}\txv_{k-1}+\A^{k-1}_{k}\txv_1\otimes \txv_{k-1}+\cdots+\A^{k-1}_{2(k-1)}\txv_{k-1}\otimes\txv_{k-1}, \notag.
\end{eqnarray}
where, $\txv_i=\xv^{[i]}$. Let $\txv=(\txv_1^\tra,\cdots,(\txv_{k-1})^\tra)^\tra$ and $\tzv_i=\txv^{[i]}$ be the variables introduced for the CL of the above system, and satisfy equations of the form
\begin{equation}\label{eq:genzti}
\dot{\tzv}_{i+1}=\tilde{\A}^{i+1}_{i}\tzv_{i}+\tilde{\A}^{i+1}_{i+1}\tzv_{i+1}+\tilde{\A}^{i+1}_{i+2}\tzv_{i+2},
\end{equation}
for appropriate, $\tilde{\A}^{i+1}_{i},\tilde{\A}^{i+1}_{i+1}$ and $\tilde{\A}^{i+1}_{i+2}$.

To prove the result, we next show that every unique equation of the form (\ref{eq:genzti}), appearing in  the CL of (\ref{eq:sysCL1}), also appears in the CL of (\ref{eq:gen}). Firstly, note that, the above system of equations (\ref{eq:sysCL1}) is equivalent to the system of equations (\ref{eq:genz}) for $i=1,\hdots,k-1$ which is the first set of equations, i.e. for $\tzv_{1}$,  appearing in the CL (\ref{eq:genzti}). Secondly, in going from $\tzv_{i}$ to $\tzv_{i+1}=\txv\otimes\tzv_{i}$, the only new powers of $\xv$ introduced in the vector $\tzv_{i+1}$ are $\xv^{[{i(k-1)+p}]}(\equiv \txv_{i(k-1)+p}), \text{where}\,\,p=1,\hdots,k-1$. From (\ref{eq:genz}), these variables satisfy,
\begin{eqnarray}
  \dot{\txv}_{i(k-1)+1} &=& \A^{i(k-1)+1}_{i(k-1)}\txv_{i(k-1)}+\A^{i(k-1)+1}_{i(k-1)+1}\txv_{i(k-1)+1}+\cdots+\A^{i(k-1)+1}_{i(k-1)+k-1}\txv_{i(k-1)+k-1}\notag \\
  &+&\A^{i(k-1)+1}_{i(k-1)+k}\txv_{1}\otimes \txv_{(i+1)(k-1)}, \notag\\
  \dot{\txv}_{i(k-1)+2} &=& \A^{i(k-1)+2}_{i(k-1)+1}\txv_{i(k-1)+1}+\A^{i(k-1)+2}_{i(k-1)+2}\txv_{i(k-1)+2}+\cdots+\A^{i(k-1)+2}_{i(k-1)+k-1}\txv_{i(k-1)+k-1}\notag \\
  &+&\A^{i(k-1)+2}_{i(k-1)+k}\txv_1\otimes \txv_{(i+1)(k-1)}+\A^{i(k-1)+2}_{i(k-1)+k+1}\txv_2\otimes \txv_{(i+1)(k-1)}, \notag \\
   &\vdots & \label{eq:genCL2}\\
  \dot{\txv}_{i(k-1)+k-1} &=& \A^{i(k-1)+k-1}_{i(k-1)+k-2}\txv_{i(k-1)+k-2}+\A^{i(k-1)+k-1}_{i(k-1)+k-1}\txv_{i(k-1)+k-1}+\cdots\notag\\
  &+&\A^{i(k-1)+k-1}_{i(k-1)+2(k-1)}\txv_{k-1}\otimes\txv_{(i+1)(k-1)} \notag.
\end{eqnarray}
Note that the term $\txv_{i(k-1)}$ appears in $\tzv_{i}$, the terms $\txv_{i(k-1)+p},p=1,\hdots,k-1$ appear in $\tzv_{i+1}$, and the terms $\txv_{p}\otimes \txv_{(i+1)(k-1)},p=1,\hdots,k-1$ appear in $\tzv_{i+2}=\txv\otimes\tzv_{i+1}$. Hence, the system (\ref{eq:genCL2}) can be written in the form,
\begin{equation}\label{eq:genztip}
\left(\begin{array}{c}
  \dot{\txv}_{i(k-1)+1}  \\
  \dot{\txv}_{i(k-1)+2} \\
  \vdots \\
  \dot{\txv}_{i(k-1)+k-1}
\end{array}\right)
=\tilde{\B}^{i+1}_{i}\tzv_{i}+\tilde{\B}^{i+1}_{i+1}\tzv_{i+1}+\tilde{\B}^{i+1}_{i+2}\tzv_{i+2},
\end{equation}
for appropriate  $\tilde{\B}^{i+1}_{i},\tilde{\B}^{i+1}_{i+1}$ and $\tilde{\B}^{i+1}_{i+2}$. For example, $\tilde{\B}^{i+1}_{i+1}$, will take the form, 
\begin{equation}\label{eq:B}
\tilde{\B}^{i+1}_{i+1}=\left(
        \begin{array}{cccccc}
        \Oz_1  & \A^{i(k-1)+1}_{i(k-1)+1} & \A^{i(k-1)+1}_{i(k-1)+2}& \A^{i(k-1)+1}_{i(k-1)+3} & \cdots & \A^{i(k-1)+1}_{i(k-1)+k-1} \\
        \Oz_2  & \A^{i(k-1)+2}_{i(k-1)+1} & \A^{i(k-1)+2}_{i(k-1)+2} & \A^{i(k-1)+2}_{i(k-1)+3} & \cdots & \A^{i(k-1)+2}_{i(k-1)+k-1} \\
        \Oz_3  & 0 & 0 & \cdots & \cdots & \cdot  \\
        \vdots  & \vdots & \vdots  & \vdots  & \vdots  & \vdots  \\
        \Oz_{k-1}  & 0 & 0 & \cdots &  \A^{i(k-1)+k-1}_{i(k-1)+k-2}& \A^{i(k-1)+k-1}_{i(k-1)+k-1} \\
        \end{array}
      \right),
\end{equation}
where, $\Oz_p$ are zero matrices of size $\nx^{i(k-1)+p}\times (\nx_k^{i+1}-(k-1))$ for $p=1,\cdots,k-1$. Expressions for $\tilde{\B}^{i+1}_{i}$ and $\tilde{\B}^{i+1}_{i+2}$ take a similar form, which we omit here for brevity. Finally, note that
$\tilde{\B}^{i+1}_{i},\tilde{\B}^{i+1}_{i+1}$ and $\tilde{\B}^{i+1}_{i+2}$ are sub-matrices defining $\tilde{\A}^{i+1}_{i},\tilde{\A}^{i+1}_{i+1}$ and $\tilde{\A}^{i+1}_{i+2}$, respectively, since the vectors on the LHS of equation~(\ref{eq:genztip}) are elements of the vector $\dot{\tzv}_{i+1}$ appearing in (\ref{eq:genzti}). The system of equations~(\ref{eq:genztip}), thus, represents a subsystem of equations appearing in the CL (\ref{eq:genzti}). Hence, by choosing  $\nt^\prime=\nt(k-1)$ for truncation level of the CL of (\ref{eq:gen}), we will obtain a system of equations representing a set of all unique equations appearing in the CL of (\ref{eq:qd}) with the truncation level $\nt$.
\end{proof}

\bibliographystyle{plain}
\bibliography{qcbib}

\end{document}